\documentclass[11pt]{amsart}
\usepackage[colorlinks]{hyperref}
\usepackage{xcolor}
\usepackage[a4paper,asymmetric]{geometry}
\usepackage{mathscinet}
\usepackage{fullpage}
\usepackage{latexsym}
\usepackage{amsthm}
\usepackage{amssymb}
\usepackage{amsfonts}
\usepackage{amsmath}
\newtheorem{theorem}{Theorem}[section]
\newtheorem{thm}[theorem]{Theorem}

\newtheorem{lem}[theorem]{Lemma}
\newtheorem{proposition}[theorem]{Proposition}

\newtheorem{corollary}[theorem]{Corollary}

\theoremstyle{definition}

\newtheorem{defn}[theorem]{Definition}
\theoremstyle{remark}

\newtheorem{rem}[theorem]{Remark}
\numberwithin{equation}{section}

 \DeclareMathAlphabet{\mathpzc}{OT1}{pzc}{m}{it}
 
 \newcommand{\E}{\mathbb{E}}            
 
 \newcommand{\e}{\varepsilon}
 \newcommand{\p}{\partial}

 \newcommand{\Ll}{\langle}
 \newcommand{\Rr}{\rangle}

 \newcommand{\N}{\mathbb{N}}
 \newcommand{\R}{\mathbb{R}}
 
 \newcommand{\PP}{\mathbb{P}}
 \newcommand{\mcl}{\mathcal}
 
 \newcommand{\Be}{\begin{equation}}
 \newcommand{\Ee}{\end{equation}}
 \newcommand{\Bs}{\begin{split}}
 \newcommand{\Es}{\end{split}}
  \newcommand{\Bes}{\begin{equation*}}
 \newcommand{\Ees}{\end{equation*}}
 \newcommand{\BT}{\begin{thm}}
 \newcommand{\ET}{\end{thm}}
 \newcommand{\Bp}{\begin{proof}}
 \newcommand{\Ep}{\end{proof}}
 \newcommand{\BL}{\begin{lem}}
 \newcommand{\EL}{\end{lem}}
 \newcommand{\BP}{\begin{proposition}}
 \newcommand{\EP}{\end{proposition}}
 \newcommand{\BC}{\begin{corollary}}
 \newcommand{\EC}{\end{corollary}}
 \newcommand{\BR}{\begin{rem}}
 \newcommand{\ER}{\end{rem}}
 \newcommand{\BD}{\begin{defn}}
 \newcommand{\ED}{\end{defn}}
 \newcommand{\BI}{\begin{itemize}}
 \newcommand{\EI}{\end{itemize}}

 \newcommand{\dif}{{\rm d}}
 \newcommand{\re}{{\rm e}}

 \newcommand{\tl}{\tilde}
\begin{document}
\title
[Smooth densities of SDEs with stable like noises] {Smooth densities of stochastic differential equations forced by degenerate stable type noises}
\author[L. Xu]{Lihu Xu}
\subjclass[2000]{}
\date{}
\maketitle
\begin{abstract}
Using the Bismut's approach to Malliavin calculus, we introduce a simplified Malliavin matrix (\cite{No86}) for stochastic differential
equations (SDEs) force by degenerate stable like noises. For the degenerate SDEs driven by Wiener noises, one can derive a Norris type lemma and use it \emph{iteratively} to prove
the smoothness of density functions. Unfortunately, Norris type lemma is very hard to be iteratively applied to SDEs with stable like noises. In this paper,
we derive a simple inequality as a replacement and use it to show that two families of
degenerate SDEs with stable like noises admit smooth density functions. One family is the linear SDEs studied by Priola and Zabczyk (\cite{PrZa09}), under some additional assumption we can iteratively use the inequality to get the smoothness of the density. The other family is the general SDEs with stable like noises, we can apply this inequality only \emph{one} time and thus derive that the SDEs admit smooth density if the \emph{first} order Lie brackets span $\R^d$. The crucial step in this paper is estimating the smallest eigenvalue of the simplified Malliavin matrix, which only uses some elementary facts of Poisson processes and undergraduate level ordinary differential equations.
\end{abstract}

 \ \\
\section{Introduction}
We are concerned with smooth densities for the degenerate stochastic
differential equations forced by stable like noises as follows:
\begin{equation} \label{e:SDE}
\begin{cases}
\dif X_t=a(X_t) \dif t+B \dif L_t, \\
X_0=x,
\end{cases}
\end{equation}
where $X_t \in \R^d$ for each $t \ge 0$, $x \in \R^d$ and the hypotheses of
$A, B, L_t$ will be stated below.
We shall introduce a simplified Malliavin matrix associated to Eq. \eqref{e:SDE} and use it to study
the smoothness of the associated transition probability densities.

As $a(x)$ is linear and the classical Kalman rank condition holds,
Priola and Zabczyk proved by Fourier analysis that transition probabilities associated to Eq. \eqref{e:SDE} admit smooth densities (\cite{PrZa09}) for a large family of
$L_t$. Under some additional assumptions on $L_t$, our results give a new proof for theirs. When $a(x)$ is a general
bounded smooth function, we show that Eq. \eqref{e:SDE} admits smooth density functions  as long as
the first order Lie brackets span $\R^d$. Our results seem to be completely new.
\vskip 3mm

Let us also compare our results with some known results on Malliavin calculus on SDEs with jump processes. \cite{BaCl11}
studied integration by parts for the jump processes with their jumps depending on the particle positions. \cite{KT01,Ta02} also studied
the density smoothness of the transition probabilities of a family of SDEs forced by
jump processes, which seems not to cover our results.
\cite{Zha12,Zha13} studied the same problems as ours for degenerate SDEs forced by symmetric $\alpha$-stable noises. For more research in this direction,
we refer to \cite{BaCr86,BiGrJa86,Bis82,Bis83,BoKu11,Cas09,IsKu06,KT01,Kus10,NoSi06,Nua06,Pic96, SoZh14}.

Let us specify our method with more details as below. For the degenerate SDEs driven by Wiener noises, one can derive a Norris type lemma and use it \emph{iteratively} to prove
the smoothness of density functions. Unfortunately, Norris type lemma is very hard to be iteratively applied to SDEs with stable like noises. In this paper,
we derive a simple (coercive) inequality as a replacement and use it to estimate the smallest eigenvalue of our simplified Malliavin matrix.
For the linear SDEs studied by Priola and Zabczyk, under some additional assumption we can use this inequality iteratively to get the smoothness of the densities.
 For the general SDEs with stable like noises, we can apply this inequality only one time and thus derive that the SDEs admit smooth density if the \emph{first} order Lie brackets span $\R^d$. The crucial step in this paper is estimating the smallest eigenvalue of the simplified Malliavin matrix, which only uses some elementary facts of Poisson processes and undergraduate level ordinary differential equations.

{\bf Acknowledgements:} The author would like to gratefully thank Martin Hairer for his encouragement and many suggestions for revising the 2013 version of this paper.
He also would like to gratefully thank Zhen-Qing Chen, Zhao Dong, Yulin Song, Tusheng Zhang and Xicheng Zhang for their helpful discussions. Special thanks are due to
Xicheng Zhang for pointing out the references \cite{BiGrJa86}-\cite{Bis83}. Part of work was done during the author visiting USTC and WHU.


\section{Some preliminary of L\'evy processes and main results}
Denote $\R^d_0=\R^d\setminus \{0\}$. Let $L_t$ be a pure jump process with c\`{a}dl\`{a}g trajectories, it is well known that
there exist a Poisson random measure $N$ on $(\R^d_0 \times \R^+, \mcl B(\R^d_0 \times \R^+))$ and a L\'evy intensity measure
$\nu$ on $(\R^d_0, \mcl B(\R^d_0))$ associated to $L_t$, such that
\Bes
\nu(\{0\})=0, \ \ \ \int_{\R^d_0} (1 \wedge |z|^2) \nu(\dif z)<\infty;
\Ees
\begin{equation} \label{e:StaDef}
L_t=\int_0^t \int_{|z| \le 1} z \tl N(\dif z, \dif s)+\int_0^t \int_{|z|>1} z N(\dif z, \dif s);
\end{equation}
where $\tl N(\dif z,\dif s)=N(\dif z, \dif s)-\nu(\dif z) \dif s$. It is well known that the random measure $N$ can be defined by:
 for all $A \in \mcl B(\R^d_0)$
\Bes
N(A \times [0,t])=\sum_{0 \le s \le t} \sharp \{L_s-L_{s-}: L_s-L_{s-} \in A\}.
\Ees
Moreover, $N(A \times [0,t])$ satisfies a Poisson distribution with the intensity $\nu(A) t$, more precisely,
$$\PP\left(N(A \times [0,t])=k\right)=\frac{(\nu(A) t)^k}{k!} \re^{-\nu(A) t} \ \ \ \ \ \ k=0,1,2,....$$
We shall use this easy relation frequently in the proof of our crucial Lemma \ref{l:MEigEst} below.
\vskip 3mm

Throughout this paper we assume that
\begin{itemize}
\item[{\bf (H1)}] $\nu$ has a density function $\rho \in C^1(\R^d_0,\R^+)$ and
there exists some $\alpha \in (0,2)$ such that
$$\rho(z)=\frac{\vartheta(z)}{|z|^{d+\alpha}} \ \ \ \forall z \in B_1 \setminus \{0\},$$
where $B_1 \setminus \{0\}=\{z \in \R^d_0: |z|<1\}$ and $\vartheta: B_1 \setminus \{0\}\rightarrow \R^+$ is a $C^1$ bounded function such that for all $z \in B_0 \setminus \{0\}$
$$c \le |\nabla \vartheta(z)| \le C, \ \ \ c \le  \vartheta(z) \le C \ \ {\rm with \ some \  constants} \ C>c>0.$$
\item[{\bf (H2)}]  $a \in C^\infty_b(\R^d,\R^d)$ is a nonzero smooth function whose all derivatives are bounded.
\item[{\bf (H3)}]  $B \in \R^{d \times d}$ is a constant matrix and $B_i$ is the $i$-th column vector of $B$ ($i=1,...,d$).
\end{itemize}
Our main results are the following two theorems.
\begin{thm} \label{t:MaiThm1}
Let $(H1)-(H3)$ all hold. Assume that there exists a nonzero matrix $A \in \R^{d \times d}$ such that
$$a(x)=Ax \ \ \ \ \forall x  \in \R^d.$$
Further assume that there exists some $n \in \N$ such that
$${\rm rank}[B,AB,...,A^nB]=d.$$
Then, for all $t>0$ the transition probability $P_t(x,.)$ associated to the solution of Eq. \eqref{e:SDE} $X_t(x)$ has a smooth density function.
\end{thm}

\begin{thm} \label{t:MaiThm2}
Let $(H1)-(H3)$ all hold.
Assume that the following uniform H\"{o}rmander condition holds:
$$\inf_{x \in \R^d} \inf_{|u|=1} \sum_{i=1}^d \left(|\Ll\nabla  a(x) B_i,u\Rr|^2+|\Ll B_i,u\Rr|^2\right)>0.$$
Then, for all $t>0$ the transition probability $P_t(x,.)$ associated to the solution of Eq. \eqref{e:SDE} $X_t(x)$ has a smooth density function.
\end{thm}

Comparing with \cite{PrZa09}, our assumption in (H1) is more strict than the one therein:
$$\inf_{|h|=1} \int_{|\Ll z,h\Rr| \le r} |\Ll z,h \Rr| \nu(\dif z) \ge r^{2-\alpha} \ \ \ {\rm for \ some \ sufficiently \ small} \ r>0.$$
Because the Skorohod integral \eqref{e:SkoInt} below includes some gradient, it seems the differentiability assumption in (H1) is needed. Our second theorem
seems to be completely new comparing with the known results. We shall denote 
$$|B|=\max_{1 \le i \le d}|B_i|.$$
\ \

\section{Integration by parts formula and simplified Malliavin matrix for jump L\'evy processes}
Denote the solution of Eq. \eqref{e:SDE} by $(X_t(x,L))_{t \ge 0}$, it is a functional of $x$ and $L$.
For any $\xi \in \R^d$ it is well known that the derivative
$\nabla_\xi X_t$ satisfies
$$\dif \nabla_\xi X_t=\nabla a(X_t) \nabla_\xi X_t \dif t, \ \ \ \nabla_\xi X_0=\xi.$$
There exists a Jacobi flow $J_{t}$ associated to Eq. \eqref{e:SDE} such that
\Be \label{e:JtEqn}
\dif J_{t}=\nabla a(X_t) J_{t}\dif t, \ \ \  J_{0}=I.
\Ee
Clearly we have
$$\nabla_\xi X_t=J_t \xi.$$
For every $t\ge 0$, $J_t$ has an inverse. We denote $K_t=J^{-1}_t$ for each $t \ge 0$ and $K_t$ satisfies
\Be \label{e:KtEqn}
\dif K_t=-K_t \nabla a(X_t) \dif t, \ \ \ K_0=I.
\Ee
\vskip 3mm
Denote $\Omega=D(\R^+,\R^d)$ the collection of function
$\omega: \R^+ \rightarrow \R^d$ which is right continuous and has left limit. In our situation, it is convenient for us to take
$\Omega=D(\R^+,\R^d)$. Let $(\mcl F_t)_{t \ge 0}$ be the canonical filtration of $\Omega$ and $\mcl P$ be the predictable
$\sigma$-field on $\R^+ \times \Omega$.
Let $v:\R^d_0 \times \R^+ \times \Omega  \rightarrow \R$ be a $\mcl B(\R^d_0) \times \mcl P$-measurable function such that
$$\E\int_0^t \int_{\R^d_0} |v(z,s)| \nu(\dif z) \dif s<\infty \ \ \ \ \forall \ t>0.$$
Define
$$V(t)=\int_0^t \int_{\R^d_0} v(z,s) N(\dif z, \dif s),$$
and
\Bes
D_V X_t=\lim_{\e \rightarrow 0}\frac{X_t(x,L+\e V)-X_t(x,L)}{\e},
\Ees
the above limit exists in $L^1((\Omega,\mcl F,\PP);\R^d)$ for each $t \ge 0$ (\cite{BaCl11}). The $D_V X_t$ satisfies
\Bes
\dif D_V X_t=\nabla a(X_t) D_V X_t \dif t+B \dif V_t, \ \ \ D_V X_0=0,
\Ees
which is solved by
\Be \label{e:DVXt}
D_V X_t=J_t \int_0^t \int_{\R^d_0} K_s B v(z,s) N(\dif z, \dif s).
\Ee

\begin{lem}
Let $\xi(t)$ be an adapted process valued on $\R^d$ such that there exist some $C_1, C_2>0$ such that
$$\sup_{\omega \in \Omega} |\xi(t,\omega)| \le C_2 \re^{C_1 t} \ \ \ \forall \ t \ge 0.$$
Let
\Be \label{e:SpeH}
h(z)=\varphi(z) |z|^{4}
\Ee
where $\varphi:\R^d \rightarrow \R^+$ is a smooth function such that $h(z)=1$ for $|z|  \le  1$ and $h(z)=0$ for $|z| \ge 2$. Take $v(z,t)=h(z) \xi(t)$ and $V(z,t)=\int_0^t v(z,s) \dif s$, then,
for all $f \in C_b^1(\R^d)$ the following relation holds:
\Be \label{e:IBPGen}
\E\left(D_V f(X_t)\right)=\E\left(f(X_t) \delta(V)\right) \ \ \ \ \ \forall t \in [0,T],
\Ee
where
\Be \label{e:SkoInt}
\delta(V)=\int_0^t \int_{\R^d_0} \frac{{\rm div}(\rho(z) h(z) \xi(s))}{\rho(z)} \tl N(\dif z, \dif s)
\Ee
 Moreover, for all $\lambda>0$ we have
\Be \label{e:ExpDelV}
\E \re^{\lambda |\delta(V)|}<C,
\Ee
\Be \label{e:CPP1}
\E \re^{\lambda \int_0^t \int_{\R^d_0} h(z) N(\dif z, \dif s)} \le C,
\Ee
where $C$ depends on $\lambda, \xi$ and $t$.
\end{lem}
\begin{proof}
\eqref{e:SpeH} is not new, we shall give a fast sketchy proof in the appendix for the completeness. For more details, one can refer to \cite{Bis82,Bis83,BiGrJa86}. Let us prove \eqref{e:ExpDelV}. It is easy to check that
$$\sup_{0 \le s \le t} \left|\frac{{\rm div}(\rho(z) h(z) \xi (s))}{\rho(z)}\right| \le c |z|^{3} \ \ \ \ \forall |z| \le 2$$
$$\sup_{0 \le s \le t} \left|\frac{{\rm div}(\rho(z) h(z) \xi(s))}{\rho(z)}\right|=0 \ \ \ \ \forall |z| \ge 2$$
where $c$ is some constant depending on $\alpha$ and $\xi$. By \cite[Theorem 25.3]{Sa90}, we immediately get the desired bound \eqref{e:ExpDelV}.
\eqref{e:CPP1} follows from \cite[Theorem 25.3]{Sa90} again.
\end{proof}

Let $\{\re_1,...,\re_d\}$ be the standard basis of $\R^d$, for $i=1,...,d$ define
$$\xi_i(t)=B^{*} K^*_t \re_i, \ \ \ v_i(z,t)=h(z) \xi_i(t),$$
by \eqref{e:DVXt} we have
$$D_{V_i}X_t=J_t \int_0^t \int_{\R^d_0} K_s B B^* K^{*}_s \re_i h(z)N(\dif z, \dif s) \ \ \ \ \forall t>0$$
with $V_i(t)=\int_0^t \int_{\R^d_0} h(z)\xi_i(s) N(\dif z, \dif s)$ for $i=1,...,d$. Therefore,
\Be \label{e:MlvJt}
[D_{V_1}X_t,...,D_{V_d}X_t]=J_t \int_0^t \int_{\R^d_0} K_s B B^* K^{*}_s h(z)N(\dif z, \dif s).
\Ee
Write
$$\mcl M_t=\int_0^t \int_{\R^d_0}  K_s BB^*K^{*}_s h(z) N(\dif z,\dif s),$$
it is called \emph{simplified Malliavin} matrix (\cite{No86}).  $\mcl M_t$ is a symmetric $d \times d$ matrix whose smallest eigenvalue $\lambda_{{\rm min}}(t)$ is
$$\lambda_{{\rm min}}(t)=\inf_{u \in \R^d: |u|=1}\Ll \mcl M_t u, u\Rr.$$
A straightforward computation gives
\Be \label{e:SmaEigRep}
\lambda_{{\rm min}}(t)=\inf_{|u|=1}\int_0^t \int_{\R^d_0} \sum_{i=1}^d |\Ll K_s B_i, u\Rr|^2 h(z) N(\dif z, \dif s).
\Ee
To prove the smoothness of densities, we need the following auxiliary lemmas.
\begin{lem} \label{l:AuxLem}
The following statements hold£º
\begin{itemize}
\item[(1)] We have $|J_t|, |K_t| \le \re^{\|\nabla a\|_\infty t}\ \ \forall t \ge 0$. In particular, $|J_t|, |K_t| \le \re^{|A| t} \ \ \forall t \ge 0$ when the condition in Theorem
\ref{t:MaiThm1} holds.
\item[(2)] Let $V_1,...,V_d$ be as above. For all $p>0, m \ge 1, T>0$ and any $(i_1,...,i_m) \in \{1,...,d\}^m$, we have
\begin{align}
&\E \sup_{0 \le t \le T}|D^m_{V_{i_1},...,V_{i_m}} X_t|^p<\infty, \label{a:DvX} \\
&\E \sup_{0 \le t \le T} |D^m_{V_{i_1},...,V_{i_m}} K_t|^p<\infty, \label{a:DvK} \\
&\E \sup_{0 \le t \le T} |D^m_{V_{i_1},...,V_{i_m}} \mcl M_t|^p<\infty.  \label{a:DvM}
\end{align}
\end{itemize}
\end{lem}

\begin{proof}
It is very easy to get (1) from Eq. \eqref{e:KtEqn} and \eqref{e:JtEqn}. By (1) and \eqref{e:DVXt},
for all $i \in \{1,...,d\}$ we have
\Be \label{e:DVXtEst}
\begin{split}
|D_{V_i} X_t| & \le \int_0^t \int_{\R^d_0} |J_t||K_s| |B|^2\re^{\|\nabla a\|_\infty s} h(z) N(\dif z, \dif s) \\
& \le \re^{3\|\nabla a\|_\infty t}|B|^2 \int_0^t \int_{\R^d_0} h(z) N(\dif z, \dif s),
\end{split}
\Ee
thus,
\Be \label{e:DVXtEst}
\begin{split}
\sup_{0 \le t \le T}|D_{V_i} X_t| & \le \re^{3\|\nabla a\|_\infty T}|B|^2 \int_0^T \int_{\R^d_0} h(z) N(\dif z, \dif s).
\end{split}
\Ee
This, together with \eqref{e:CPP1}, implies
\Be \label{e:ExpDVXt}
\E \re^{\lambda \sup_{0 \le t \le T} |D_{V_i} X_t|}<\infty \ \ \ \ \forall \lambda>0,
\Ee
from which the first inequality in (2) for $m=1$ follows immediately.

A straightforward computation gives
\Bes
\begin{split}
\dif D_{V_i V_j}^2 X_t=\nabla a(X_t) & D_{V_i V_j}^2 X_t \dif t+\nabla^2 a(X_t) D_{V_i}X_t D_{V_j} X_t \dif t \\
&+\int_{\R^d_0} (B K_t)^* \re_i \nabla h(z) (B K_t)^* \re_j h(z) N(\dif z,\dif t) \ \ \ \ \forall (i,j) \in \{1,...,d\}^2
\end{split}
\Ees
with $D_{V_i V_j}^2 X_t=0$, from which it is easy to see
\Bes
\begin{split}
|D_{V_i V_j}^2 X_t|&\le \left|\int_0^t \re^{\int_s^t \nabla a(X_r) \dif r} \nabla^2 a(X_s) D_{V_i}X_s D_{V_j} X_s \dif s\right|\\
&\ \ \ +\left|\int_0^t \re^{\int_s^t \nabla a(X_r) \dif r} \int_{\R^d_0}(B K_s)^* \re_i \nabla h(z) (B K_s)^* \re_j h(z) N(\dif z,\dif s)\right| \\
& \le I_1+I_2,
\end{split}
\Ees
where
\Be
\begin{split}
& I_1(t)=\int_0^t \re^{\|\nabla a\|_\infty (t-s)} \|\nabla^2 a\|_\infty  |D_{V_i}X_s| |D_{V_j} X_s| \dif s, \\
& I_2(t)=\int_0^t \re^{\|\nabla a\|_\infty (t-s)} |B|^2  \re^{2\|A\|_\infty s} \int_{\R^d_0} |\nabla h(z)| h(z) N(\dif z,\dif s).
\end{split}
\Ee
Thanks to \eqref{a:DvX} for $m=1$, for all $p>0$ we have
$$\E \sup_{0 \le t \le T} |I_1(t)|^p \le \re^{\|\nabla a\|_\infty T} \|\nabla^2 a\|_\infty \int_0^T   \left(\E|D_{V_i}X_s|^{2p}\right)^{\frac 12} \left(\E|D_{V_j}X_s|^{2p}\right)^{\frac 12}\dif s<\infty.$$
Observe
\Be \label{e:I2Est}
\sup_{0 \le t \le T} I_2(t) \le \re^{2\|\nabla a\|_\infty T}  |B|^2\int_0^T \int_{\R^d_0}  |\nabla h(z)| h(z) N(\dif z,\dif s),
\Ee
in view of \eqref{e:SpeH}, we have $\int_{\R^d_0}|\nabla h(z)| h(z) \nu(\dif z)<\infty$, thus
\Be \label{e:ExpGraHEst}
\E \re^{\lambda \int_0^t \int_{\R^d_0} |\nabla h(z)|  h(z) N(\dif z, \dif s)}<\infty \ \ \ \ \forall \lambda>0,
\Ee
which, together with \eqref{e:I2Est}, implies
\Be 
\E \re^{\lambda \sup_{0 \le t \le T} |I_2(t)|}<\infty \ \ \ \ \forall \lambda>0.
\Ee
The estimates about $I_1$ and $I_2$ immediately give \eqref{a:DvX} for $m=2$.
By a similar (but more tedious) argument we get \eqref{a:DvX} for $m=3,4....$

For \eqref{a:DvK}, we can prove it by a similar argument as for \eqref{a:DvX}. It remains to prove
\eqref{a:DvM}. An easy computation gives
\Bes
D_{V_i} \mcl M_t=J_1(t)+J_2(t)+J_3(t),
\Ees
where
\Bes
\begin{split}
& J_1(t)=\int_0^t \int_{\R^d_0} D_{V_i} K_s BB^*K^{*}_s h(z) N(\dif z,\dif s), \\
& J_2(t)=\int_0^t \int_{\R^d_0} K_s BB^*(D_{V_i}K_s)^* h(z) N(\dif z,\dif s), \\
& J_3(t)=\int_0^t \int_{\R^d_0} K_s B B^*K^{*}_s \nabla h(z) h(z) B^*K^{*}_s \re_i N(\dif z,\dif s).
\end{split}
\Ees
It is easy to see that for all $t \in (0,T]$
\Bes
\begin{split}
|J_1(t)| &\le \int_0^t \int_{\R^d_0} |D_{V_i} K_s| |B|^2 \re^{\|\nabla a\|_\infty s} h(z) N(\dif z,\dif s) \\
& \le   |B|^2 \re^{\|\nabla a\|_\infty T} \sup_{0 \le t \le T} |D_{V_i} K_t|\int_0^T \int_{\R^d_0} h(z) N(\dif z,\dif s),
\end{split}
\Ees
combining the above inequality with \eqref{a:DvX} and \eqref{e:CPP1}, by H\"{o}lder inequality we immediately get
\Be \label{e:J1}
\E \sup_{0 \le t \le T} |J_1(t)|^p<\infty \ \ \ \ \forall p>0.
\Ee
By the same method, we have
\Be \label{e:J2}
\E \sup_{0 \le t \le T} |J_2(t)|^p<\infty \ \ \ \ \forall p>0.
\Ee
For $J_3$, by a similar argument as above we have for all $t \in (0,T]$
\Bes
\begin{split}
|J_3(t)|
& \le  |B|^3 \re^{3\|\nabla a\|_\infty T}\int_0^T \int_{\R^d_0}|\nabla h(z)| h(z) N(\dif z,\dif s),
\end{split}
\Ees
which, together with \eqref{e:ExpGraHEst}, immediately gives
\Be \label{e:J3}
\E \sup_{0 \le t \le T} |J_3(t)|^p<\infty \ \ \ \ \forall p>0.
\Ee
Collecting the estimates for $J_1, J_2, J_3$, we immediately get \eqref{a:DvM} for $m=1$.
By a similar (but more tedious) argument we get the inequalities in (3) for $m=2,3,....$
\end{proof}

The next lemma is a criterion for the smoothness of the density, which will be used to prove our main results.
\begin{lem} \label{l:SmoCri}
If $\mcl M_t$ is invertible a.s. for all $t>0$ and further satisfies
$$\E |\mcl M^{-1}_t|^p<\infty \ \ \ \ \forall p>0.$$
Then, for all $t>0$ the transition probability $P_t(x,.)$ associated to the solution of Eq. \eqref{e:SDE} $X_t(x)$ has a smooth density function.
\end{lem}

\begin{proof}
To prove the smoothness of the density, it suffices to show that for all $f \in C^\infty_b(\R^d)$ we have
\Be \label{e:GraBou}
\bigg|\E\left(\nabla^m_{i_1,...,i_m} f(X_t)\right)\bigg| \le C \|f\|_\infty \ \ \ \ \
\forall m \ge 1 \ \ \ \forall (i_1,...,i_m) \in \{1,...,d\}^m,
\Ee
where $\nabla^m_{i_1,...,i_m} =\frac{\p^m}{\p x_{i_1}...\p x_{i_m}}$ and $C$ depends on $t$ and $(i_1,...,i_m)$.

For the notational simplicity, write
$$V(t)=[V_1(t),...,V_d(t)],\ \ \ \ D_V X_t=[D_{V_1} X_t,..., D_{V_d} X_t],$$
they are both $d \times d$ matrices. It is clear to see from \eqref{e:MlvJt}
$$D_V X_t=J_t \mcl M_t.$$
By the relation $D_V f(X_t)=\nabla f(X_t) D_V X_t$, we get
$$\nabla f(X_t)=D_V f(X_t) \mcl M_t^{-1} K_t$$
and thus
$$\nabla_i f(X_t)=\sum_{j=1}^d D_{V_j} f(X_t) (\mcl M_t^{-1} K_t)_{ji} \ \ \ i=1,...,d.$$
It is easy to see that
\Be \label{e:NabFi}
\E\left(\nabla_i f(X_t)\right)
=\sum_{j=1}^d \left\{\E \left[D_{V_j} \left(f(X_t) (\mcl M_t^{-1} K_t)_{ji}\right)\right]-\E \left[f(X_t) D_{V_j}\left((\mcl M_t^{-1} K_t)_{ji}\right)
\right]\right\}.
\Ee
Using integration by parts \eqref{e:IBPGen} and H\"{o}lder inequality we have
\Be \label{e:GraEst1}
\begin{split}
\left|\E \left[D_{V_j} \left(f(X_t) (\mcl M_t^{-1} K_t)_{ji}\right)\right]\right|
& \le \|f\|_\infty \|K_t\|_\infty \left(\E|\mcl M^{-1}_t|^2\right)^{\frac 12} \left(\E|\delta(V_j)|^2\right)^{\frac 12} \\
\end{split}
\Ee
Moreover, we have
\Bes
\begin{split}
D_{V_j}\left((\mcl M_t^{-1} K_t)_{ji}\right)&=\left(D_{V_j} \mcl M_t^{-1} K_t\right)_{ji}+\left(\mcl M_t^{-1} D_{V_j} K_t\right)_{ji} \\
&=\left(\mcl M_t^{-1} D_{V_j} \mcl M_t\mcl M_t^{-1} K_t\right)_{ji}+\left(\mcl M_t^{-1} D_{V_j} K_t\right)_{ji},
\end{split}
\Ees
this, together with H\"{o}lder inequality, implies
\Be \label{e:GraEst2}
\begin{split}
\left|\E \left[f(X_t) D_{V_j}\left((\mcl M_t^{-1} K_t)_{ji}\right)
\right]\right| & \le \|f\|_\infty \|K_t\|_\infty\left(\E|\mcl M_t^{-1}|^4\right)^{\frac 12}\left(\E |D_{V_j}M_t|^2\right)^{\frac12} \\
&\ \ +\|f\|_\infty \left(\E|\mcl M_t^{-1}|^2\right)^{\frac 12}\left(\E |D_{V_j}K_t|^2\right)^{\frac12}
\end{split}
\Ee
Combining \eqref{e:NabFi}-\eqref{e:GraEst2}, by Lemma \ref{l:AuxLem} and the assumption we have
$$\left|\E\left(\nabla_i f(X_t)\right)\right| \le C \|f\|_\infty \ \ \ \forall i \in \{1,...,d\},$$
where $C$ depends on $t, i$.

A straightforward computation gives
\Be
\begin{split}
\nabla^2 f(X_t)&=\mcl M_t^{-1} K_t \left(D^2_V f(X_t)-\nabla f(X_t) D^2_V X_t\right) \mcl M_t^{-1} K_t \\
&=\mcl M_t^{-1} K_t \left(D^2_V f(X_t)-D_V f(X_t) \mcl M_t^{-1} K_t D^2_V X_t\right) \mcl M_t^{-1} K_t,
\end{split}
\Ee
using integration by parts and H\"{o}lder inequality, by Lemma \ref{l:AuxLem} and Corollary \ref{c:InvM} we
get
$$\left|\E\left(\nabla^2_{ij} f(X_t)\right)\right| \le C \|f\|_\infty \ \ \ \forall (i,j) \in \{1,...,d\}^2,$$
where $C$ depends on $i,j$ and $t$.

Iteratively using the same argument as above, we finally get the desired \eqref{e:GraBou}.
\end{proof}

\section{Proof of Theorem \ref{t:MaiThm1}} \label{s:Lin}
When $a(x):=Ax$ is linear, we have
$$J_t=\re^{At}, \ \  K_t=\re^{-At}.$$
\begin{lem} \label{l:UvUpB}
Let $u, v \in \R^d$ both be nonzero vectors with some $p>0$ such that
\Be
\Ll v,u\Rr \ge p   \ \ ({\rm or} \ \Ll u,v\Rr \le -p).
\Ee
Then there exist some $\theta=\frac{1}{2|u||v||A|}\re^{-|A|}$ and
$$\delta=\left(\theta p\right) \wedge 1$$ such that for all $t  \in (0,\delta)$.
\Be \label{e:UniUpB}
\Ll K_t v,u\Rr\ge p/2   \ \ ({\rm respectively} \ \Ll K_t v,u\Rr \le -p/2).
\Ee
Moreover, for all $v \in \R^d$ the following relation holds: for all $l \ge 1$,
\Be \label{e:KalRel}
K_t v=\sum_{j=0}^{l-1} \frac{(-t)^j}{j!} A^{j} v+(-1)^{l} \int_0^t \int_0^{s_1}...\int_0^{s_{l-1}} K_{s_{l}} A^{l} v
\dif s_{k}...\dif s_1.
\Ee
\end{lem}
\begin{proof}
Differentiating $K_t$ with respect to $t$, we get
$$\frac{\dif K_t}{\dif t}=-K_t A,$$
thus for all $t \in (0,1)$,
\Be
|\Ll K_t v,u\Rr-\Ll v, u\Rr| \le \int_0^t |A| \re^{|A| s}|u||v| \dif s \le t |u||v||A| e^{|A|}.
\Ee
Therefore, we get
\Be
|\Ll K_t u,v\Rr-\Ll u, v\Rr| \le p/2 \ \ \ \ \forall t \in (0,\delta).
\Ee
This immediately implies the first inequality.

For each $j \ge 0$, differentiating $K_t A^{j}v$ with respect to $t$, we obtain
$$\frac{\dif}{\dif t} K_t A^{j} v=-K_t A^{j+1} v.$$
Iteratively applying above equation gives \eqref{e:KalRel}.
\end{proof}
\begin{rem}
The inequality \eqref{e:UniUpB} is a replacement of Norris Lemma in our special situation. Thanks to \eqref{e:KalRel},
we can use this inequality \eqref{e:UniUpB} iteratively.
\end{rem}

Let us now prove the following crucial lemma.
\begin{lem} \label{l:MEigEst}
Assume the conditions in Theorem \ref{t:MaiThm1} hold. For any $\gamma>0$ and $\ell \in (0,1/4)$, there exist some $\e_0>0$
depending on $\gamma$, $\ell$ and some $t_0 \in (0,1)$ depending on $\e_0$
such that $\lim_{\e_0 \rightarrow 0} t_0=0$ and that for all
$\e \in (0,\e_0)$ and $t \ge t_0$ we have
\Be \label{e:MEigEst}
\PP\left(\lambda_{{\rm min}}(t) \le \e\right) \le C\re^{-c \left(\e^{\alpha \ell} |\log \e|^{\gamma}\right)^{-1}}.
\Ee
where $c$ only depends on $|A|$, $|B|$ and $C$ depends on $|A|$, $|B|$, $t$.
\end{lem}

\begin{proof}
Our proof follows the spirit in \cite{No86}.
Write
$$\Lambda(t,u,\e^{\ell})=\int_0^t \int_{|z| \ge \e^{\ell}}\sum_{i=1}^d |\Ll K_s B_i, u\Rr|^2 h(z) N(\dif z, \dif s),$$
by \eqref{e:SmaEigRep}, to prove the desired inequality, it suffices to show that there exist some $\e_0>0$ depending on $\ell,\gamma$ and some $t_0$ depending on
$\e_0$ such that $\lim_{\e_0 \rightarrow 0} t_0=0$ and that for all
$\e \in (0,\e_0]$ and $t \ge t_0$,
\Bes
\PP\left(\inf_{|u|=1} \Lambda(t,u,\e^{\ell})\le \e\right) \le C\re^{-c \left(\e^{\alpha \ell} |\log \e|^{\gamma}\right)^{-1}}.
\Ees
Since $\Lambda(t,u,\e^{\ell})$ is increasing with respect to $t$, it suffices to prove
\Be \label{e:NewTar}
\PP\left(\inf_{|u|=1} \Lambda(t,u,\e^{\ell})\le \e\right) \le C\re^{-c \left(\e^{\alpha \ell} |\log \e|^{\gamma}\right)^{-1}} \ \ \ \ \forall \e \in (0,\e_0]
\ \forall t \in [t_0,1].
\Ee

\noindent Let us prove \eqref{e:NewTar} in the following three steps.
\vskip 3mm
\underline{\emph{Step 1}:}
Write
$$N_{t,h}=\int_0^t \int_{\R^d_0} h(z)N(\dif z,\dif s),$$
$$N_{t, \e^{\ell},h}=\int_0^t \int_{|z| \ge \e^{\ell}} h(z)N(\dif z,\dif s),$$
it is clear $N_{t, \e^{\ell},h} \le N_{t, h}$. By \eqref{e:CPP1} and Chebyshev inequality we have
\Be \label{e:CheInq}
\PP\left(N_{t, \e^{\ell},h}>M\right) \le \PP\left(N_{t,h}>M\right) \le C\re^{-M} \ \ \ \forall M>0,
\Ee
where $C$ depends on $t$.

Taking $\eta=\frac{\re^{-2|A|}}{2d|B|^2} \frac{\e}{M}$, by (1) of Lemma \ref{l:AuxLem},
we easily get that for all $u,v \in \mathbb{S}^{d-1}$ with $|u-v| \le \eta$,
\Bes
\bigg|\sum_{i=1}^d|\Ll K_s B_i, u\Rr|^2-\sum_{i=1}^d|\Ll K_s B_i, v\Rr|^2\bigg| \le \frac{\e}M \ \ \ \forall s \in [0,1].
\Ees
Hence, as $N_{t, \e^{\ell},h} \le M$ we have
\Be \label{e:UniCon}
|\Lambda(t,u,\e^{\ell})-\Lambda(t,v,\e^{\ell})|  \le \e \ \ \ \ \forall t \in [0,1].
\Ee
By the compactness, $\mathbb{S}^{d-1}$ has a finite open sets cover
$(\mcl U_k)_{1 \le k \le W}$
such that
$W \le C_d \left(M\e^{-1}\right)^{d-1}$ with $C_d$ only depending on $d$
and that the diameter of each open set $\mcl U_k$ is $\eta$.
\vskip 2mm
Take any $u_k \in \mcl U_k$ for all $k$, it is easy to see from \eqref{e:UniCon} that for all
$t \in [0,1]$ we have
$$\left\{\inf_{|u|=1} \Lambda(t,u,\e^{\ell})\le \e, N_{t, \e^{\ell}, h} \le M\right\}
\subset \bigcup_{k=1}^{W} \left\{\Lambda(t,u_k,\e^{\ell})\le 2\e, N_{t, \e^{\ell},h} \le M\right\}$$
and thus
\Be \label{e:SplBal}
\begin{split}
\PP\left(\inf_{|u|=1} \Lambda(t,u,\e^{\ell})\le \e, N_{t, \e^{\ell}, h} \le M\right)
& \le \sum_{k=1}^W \PP\left(\Lambda(t,u_k,\e^{\ell})\le 2\e, N_{t, \e^{\ell}, h} \le M\right) \\
& \le C_d (M \e^{-1})^{d-1} \sup_{u \in \mathbb{S}^{d-1}} \PP\left(\Lambda(t,u,\e^{\ell})\le 2\e\right).
\end{split}
\Ee
\vskip 3mm

\underline{\emph{Step 2}:} We shall prove in the step 3 below that for any $\gamma>0$ and $\ell \in (0,1/4)$, there exist some $\e_0>0$
depending on $\gamma$, $\ell$ and some $t_0 \in (0,1)$ depending on $\e_0$
such that $\lim_{\e_0 \rightarrow 0} t_0=0$ and that for all
$\e \in (0,\e_0)$ and $t \ge t_0$ we have
\Be \label{e:Tar2}
\begin{split}
\PP\left(\Lambda(t,u,\e^{\ell})\le 2\e\right) \le \re^{-c|\log \e|^{-\gamma} \nu(\e^{\ell} \le |z| \le 1)}
\end{split}
\Ee
for all $u \in \mathbb{S}^{d-1}$,
where $c>0$ only depends on $|A|$ and $|B|$.

Now we use the inequalities in the step 1 and \eqref{e:Tar2} to prove the desired \eqref{e:NewTar}.
By \eqref{e:CheInq} with $M=\frac1{\e^2}$ therein, we get
\Be
\PP\left(N_{t, \e^{\ell}, h}>\frac 1{\e^2}\right) \le C\re^{-1/\e^2}.
\Ee
This, together with \eqref{e:Tar2} and \eqref{e:SplBal}, implies
\Be
\begin{split}
\PP\left(\inf_{|u|=1} \Lambda(t,u,\e^{\ell})\le \e\right) &
\le \PP\left(\inf_{|u|=1} \Lambda(t,u,\e^{\ell})\le \e, N_{t, \e^{\ell}, h} \le \frac 1{\e^2}\right)+
\PP\left(N_{t, \e^{\ell}, h}>\frac 1{\e^2}\right) \\
 &\le C\re^{-1/\e^2}+C_d\e^{-3(d-1)} \re^{-c|\log \e|^{-\gamma} \nu(\e^{\ell} \le |z| \le 1)}.
\end{split}
\Ee
Tuning the number $c$ to be smaller and using the assumption (H1), we immediately obtain the desired inequality \eqref{e:NewTar}.
\vskip 3mm

\underline{\emph{Step 3}:}
It remains to show \eqref{e:Tar2}.
From the rank condition in Theorem \ref{t:MaiThm1}, there exist some $j_0 \le n$, $i_0  \le d$ and some constant $\kappa_0>0$ such that
\Be \label{e:ABijUpB}
|\Ll A^{j_0} B_{i_0}, u\Rr| \ge \kappa_0.
\Ee
Without loss of generality, we assume that $j_0 \ge 1$.
Denote $\theta=\frac{\re^{-|A|}}{2|A||B|}$ and choose a small number $\e_0 \in (0,1/4)$ satisfying the following conditions:
\begin{align}
& |\log \e_0|^{{-\gamma(4n)^{-n}}}<\min\{1/\theta,\kappa_0, 1/2\},  \label{e:E01}\\
& |\log \e_0|^{-2\gamma} h(\e_0^{\ell})>8\e \ \ \ \ \ \forall \e \in (0,\e_0],  \label{e:E02}\\
& |\log \e_0|^{-(4n)^{-n} \gamma} \le \min_{1 \le j \le n} \left(\frac{2^{-j-3}\theta^j}{j!} \left(\frac{\theta}{1+\theta}\right)^j\right), \label{e:E03} \\
& \inf_{1 \le j \le n} \left(\frac{\theta^j}{2^{j+1} j!}\right)^2|\log \e_0|^{-\frac{2(j+1)\gamma}{(4n)^j}} h(\e)>8\e \ \ \ \ \ \forall \e \in (0,\e_0].  \label{e:E04}
\end{align}
It is easy to check that as $\e_0$ is sufficiently small the conditions \eqref{e:E01} and \eqref{e:E03} both hold.
 \eqref{e:E02} and \eqref{e:E04} follow from \eqref{e:SpeH} and the assumption $\ell \in (0,1/4)$ for sufficiently small $\e_0$.

 We choose $t_0=\max\{\delta, \tl \delta\}$ with $\delta$ and $\tl \delta$ defined by \eqref{e:Delta} and \eqref{e:TlDelta} respectively.
 It is clear that $\lim_{\e_0 \rightarrow 0} t_0=0$.
 Now we prove \eqref{e:Tar2} by considering the following two cases. The conditions \eqref{e:E01} and \eqref{e:E02} will be used in the Case 1 below, while \eqref{e:E03} and \eqref{e:E04} will be used in Case 2.

\emph{Case 1:} $|\Ll B_{i_0}, u\Rr| \ge |\log \e_0|^{-\gamma}$.
Choose
\Be \label{e:Delta}
\delta=\theta |\log \e_0|^{-\gamma},
\Ee
thanks to
\eqref{e:E01} we have $\delta<1$. By Lemma \ref{l:UvUpB} we get
$$|\Ll K_s B_{i_0}, u\Rr|\ge \frac 12 |\log \e_0|^{-\gamma} \ \ \ \  \forall s \in (0, \delta].$$
Write $N_{t,\e^{\ell}}=\int_0^t \int_{\e^{\ell} \le |z| \le 1} N(\dif z, \dif s)$, it follows from the above inequality and \eqref{e:SpeH}
that for all $\e \in (0,\e_0]$
\Be \label{e:UvUpB1}
\begin{split}
\int_0^\delta \int_{\e^{\ell} \le |z| \le 1} |\Ll K_s B_{i_0}, u\Rr|^2 h(z) N(\dif z,\dif s)
& \ge \frac 14\int_0^\delta \int_{\e^{\ell} \le |z| \le 1} |\log \e_0|^{-2\gamma} h(z)  N(\dif z,\dif s)   \\
& \ge  \frac 14 |\log \e_0|^{-2\gamma} h(\e^{\ell}) N_{\delta, \e^{\ell}}.
\end{split}
\Ee
A straightforward computation gives
\Be \label{e:0Jum}
\PP(N_{\delta, \e^{\ell}}=0)=\re^{-\delta \nu(\e^{\ell} \le |z| \le 1)}.
\Ee
As $N_{\delta, \e^{\ell}} \ge 1$, \eqref{e:UvUpB1} and \eqref{e:E02} imply
\Bes
\int_0^\delta \int_{\e^{\ell} \le |z| \le 1} |\Ll K_s B_{i_0} v\Rr|^2 h(z) N(\dif z,\dif s) \ge
\frac 14 |\log \e_0|^{-2\gamma} h(\e^{\ell}) >2\e, \ \ \ \forall \e \in (0,\e_0].
\Ees
Hence,
\Be \label{e:1Jum}
\PP\left(\int_0^\delta \int_{\e^{\ell} \le |z| \le 1} |\Ll K_s B_{i_0}, u\Rr|^2 h(z) N(\dif z,\dif s) \le 2\e, N_{\delta, \e^{\ell}} \ge 1\right)=0.
\Ee
By \eqref{e:0Jum}, \eqref{e:1Jum} and the fact $t_0>\delta$, we have that for all $t \ge t_0$,
\Be \label{e:C1Inq}
\begin{split}
& \ \ \  \PP\left(\int_0^t \int_{\e^{\ell} \le |z| \le 1} |\Ll K_s B_{i_0},u\Rr|^2 h(z) N(\dif z,\dif s) \le
2\e\right)  \\
& \le \PP\left(\int_0^\delta \int_{\e^{\ell} \le |z| \le 1} |\Ll K_s B_{i_0},u\Rr|^2 h(z) N(\dif z,\dif s) \le
2\e\right)  \le \re^{-\delta \nu(\e^{\ell} \le |z| \le 1)}.
\end{split}
\Ee
By the definition of $\Lambda(t,u,\e^\ell)$ and $\theta |\log \e|^{-\gamma} \le \delta$,
the above inequality immediately implies the desired inequality \eqref{e:Tar2}.

\emph{Case 2:} $|\Ll B_{i_0}, u\Rr|<|\log \e_0|^{-\gamma}$.
Define
\Bes
\begin{split}
l=\inf\big\{j \ge 1: \ |\Ll A^{k} B_{i_0}, u\Rr|&<|\log \e_0|^{-\gamma(4n)^{-k}}, \ 0 \le k \le j-1; \\
& \ \ \ \ \ |\Ll A^{j} B_{i_0}, u\Rr| \ge |\log \e_0|^{{-\gamma(4n)^{-j}}}\big\},
\end{split}
\Ees
where $n$ is the constant in Theorem \ref{t:MaiThm1}. By \eqref{e:ABijUpB} and \eqref{e:E01}, as $\e_0$ is sufficiently small we have
$$|\log \e_0|^{{-\gamma(4n)^{-j}}} \le  |\log \e_0|^{{-\gamma(4n)^{-n}}} \le \kappa_0 \le |\Ll A^{j_0} B_{i_0}, u\Rr|.$$
This and the condition $|\Ll B_{i_0}, u\Rr|<|\log \e_0|^{-\gamma}$ imply
$$1 \le l \le j_0.$$

\noindent Choose
\Be \label{e:TlDelta}
\tl \delta=\theta |\log \e_0|^{{-\gamma(4n)^{-l}}}
\Ee
it is obvious $\tl \delta \le t_0$ by the definition of $t_0$ above. By Lemma \ref{l:UvUpB} and the definition of $l$, we have for all $s \in (0,\tl \delta]$
$$\Ll K_{s} A^{l} B_{i_0},u\Rr \ge \frac 12 |\log \e_0|^{{-\gamma(4n)^{-l}}} \ \ \ {\rm or} \ \ \
\Ll K_{s} A^{l} B_{i_0}, u\Rr \le -\frac 12 |\log \e_0|^{{-\gamma(4n)^{-l}}}.$$
The above two inequalities imply
\Be \label{e:IntUpB}
\left|\int_0^t \int_0^{s_1}...\int_0^{s_{l-1}} \Ll K_{s_{l}} A^{l} B_{i_0}, u\Rr
\dif s_{l}...\dif s_1\right| \ge \frac {t^{l}}{2l!} |\log \e_0|^{{-\gamma(4n)^{-l}}} \ \ \ \forall t \in (0,\tl \delta].
\Ee
From the definition of $l$ again, we have
\Be \label{e:DefLCns}
\left|\frac{(-t)^j}{j!} \Ll A^{j} B_{i_0}, u\Rr\right| < \frac{t^{j}}{j!} |\log \e_0|^{{-\gamma(4n)^{-j}}}
\ \ \ \ \forall 0 \le j \le l-1.
\Ee
Applying \eqref{e:KalRel}, by \eqref{e:IntUpB} and \eqref{e:DefLCns} we get
\Bes
\begin{split}
|\Ll K_t B_{i_0},u\Rr| & \ge \frac {t^{l}}{2l!} |\log \e_0|^{{-\gamma(4n)^{-l}}}
-\sum_{j=0}^{l-1} \frac{t^{j}}{j!} |\log \e_0|^{{-\gamma(4n)^{-j}}} \ \ \ \forall t \in (0,\tl \delta].
\end{split}
\Ees
For all $t \in [\tl \delta/2, \tl \delta]$ we have
\Be \label{e:KtBu}
\begin{split}
|\Ll K_t B_{i_0},u\Rr| & \ge \frac {\left(\frac{\tl \delta}2\right)^{l}}{2l!} |\log \e_0|^{{-\gamma(4n)^{-l}}}
-\sum_{j=0}^{l-1} \frac{\tl \delta^{j}}{j!} |\log \e_0|^{{-\gamma(4n)^{-j}}} \\
&=\frac {\theta^{l}}{2^{l+1} l!} |\log \e_0|^{{-\frac{l+1}{(4n)^{l}} \gamma}}
-\sum_{j=0}^{l-1} \frac{\theta^j}{j!} |\log \e_0|^{{-\frac{j+(4n)^{l-j}}{(4n)^{l}} \gamma}}.
\end{split}
\Ee
Observing
\Bes
\begin{split}
\sum_{j=0}^{l-1} \frac{\theta^j}{j!} |\log \e_0|^{{-\frac{j+(4n)^{l-j}}{(4n)^{l}} \gamma}}& \le (\theta+1)^l |\log \e_0|^{{-\frac{l+1}{(4n)^{l}} \gamma}}
 \sum_{j=0}^{l-1} |\log \e_0|^{{-\frac{(4n)^{l-j}-(l-j)-1}{(4n)^{l}} \gamma}}
\end{split}
\Ees
and
$$(4n)^{l-j}-(l-j)-1 \ge (l-j) \ \ \ \ \ \ \forall j<l,$$
we get 
\Be \label{e:SumLogE}
\begin{split}
\sum_{j=0}^{l-1} \frac{\theta^j}{j!} |\log \e_0|^{{-\frac{j+(4n)^{l-j}}{(4n)^{l}} \gamma}}
& \le (\theta+1)^l |\log \e_0|^{{-\frac{l+1}{(4n)^{l}} \gamma}}
\sum_{j=0}^{l-1}  |\log \e_0|^{{-\frac{l-j}{(4n)^{l}} \gamma}} \\
& \le 2(\theta+1)^l |\log \e_0|^{{-\frac{l+2}{(4n)^{l}} \gamma}},
\end{split}
\Ee
where the last inequality is by \eqref{e:E01}.
It follows from \eqref{e:E03} that
$$\frac {\theta^{l}}{2^{l+1} l!}-2(\theta+1)^l |\log \e_0|^{{-\frac{1}{(4n)^{l}} \gamma}} \ge \frac {\theta^{l}}{2^{l+2} l!},$$
which, together with \eqref{e:SumLogE} and \eqref{e:KtBu}, gives
$$|\Ll K_t B_{i_0},u\Rr| \ge \frac {\theta^{l}}{2^{l+2} l!} |\log \e_0|^{{-\frac{l+1}{(4n)^{l}} \gamma}} \ \ \ \ \forall t \in [\tl \delta/2, \tl \delta].$$
By the same argument as in the case 1, we have
\Bes
\PP\left(\int_{\tl \delta/2}^{\tl \delta} \int_{\e^{\ell} \le |z| \le 1} |\Ll K_t B_{i_0},u\Rr|^2 h(z) N(\dif z,\dif s) \le 2\e\right) \le
\re^{-\frac{\tl \delta}2 \nu(\e^{\ell} \le |z| \le 1)} \ \ \ \forall \e \in (0,\e_0],
\Ees
hence, for all $t \ge t_0$ (recall $t_0 \ge \tl \delta$) we have
\Be \label{e:C2Inq}
\PP\left(\int_{0}^{t} \int_{\e^{\ell} \le |z| \le 1} |\Ll K_t B_{i_0},u\Rr|^2 h(z) N(\dif z,\dif s)\le 2\e\right) \le
\re^{-\frac{\tl \delta}2 \nu(\e^{\ell} \le |z| \le 1)} \ \ \ \forall \e \in (0,\e_0].
\Ee
In view of $\tl \delta \ge \delta$ and $\delta=\theta |\log \e_0|^{-\gamma}$, it follows from the previous inequality that for $\forall \e \in (0,\e_0]$
\Bes 
\PP\left(\int_{0}^{t} \int_{\e^{\ell} \le |z| \le 1} |\Ll K_t B_{i_0},u\Rr|^2 h(z) N(\dif z,\dif s) \le 2\e\right) \le
\re^{-\frac{\delta}2 \nu(\e^{\ell} \le |z| \le 1)} \le  \re^{-\frac12 \theta |\log \e|^{-\gamma}\nu(\e^{\ell} \le |z| \le 1)}
\Ees
and thus the desired \eqref{e:Tar2}.
\end{proof}

\begin{proof}
By Lemma \ref{l:SmoCri}, it suffices to show $\mcl M_t$ is invertible a.s. and
$$\E |\mcl M^{-1}_t|^p<\infty \ \ \ \ \forall p>0.$$
Take any $t_1>0$ and fix it. From Lemma \ref{l:MEigEst} we can choose $\e_0>0$ sufficiently small so that
$$\delta \le t_1/2$$
and that \eqref{e:MEigEst} holds for $t>\delta$ (in particular for $t=t_1$). Taking $\e=1/n$ in \eqref{e:MEigEst} and
writing $E_n=\{\lambda_{{\rm min}}(t_1) \le 1/n\}$, we have
$$\sum_{n=n_0}^\infty \PP\left(E_n\right) \le \sum_{n=n_0}^\infty C\re^{-c \left(n^{\alpha \ell} |\log n|^{\gamma}\right)}<\infty,$$
where $n_0=[1/\e_0]+1$. By Borell-Cantelli Lemma we have $\lambda_{{\rm min}}(t_1)>0$ a.s. and thus $\mcl M_{t_1}$ is invertible a.s..
We take the largest
eigenvalue of $\mcl M^{-1}_{t_1}$ i.e. $(\lambda_{{\rm min}}(t_1))^{-1}$ as $|\mcl M^{-1}_{t_1}|$ (recall all the norms of a finite
dimension space are equivalent), \eqref{e:MEigEst} implies
$$\PP(|\mcl M^{-1}_{t_1}| \ge 1/\e) \le C\re^{-c \left(\e^{\alpha \ell} |\log \e|^{\gamma}\right)^{-1}} \ \ \ \forall \e \in (0,\e_0],$$
which immediately implies the desired inequality for $t=t_1$.
Since $t_1>0$ is arbitrary, the proof is completed.
\end{proof}

\section{Proof of Theorem \ref{t:MaiThm2}}
To prove Theorem \ref{t:MaiThm2}, we shall use the same procedure as proving Theorem \ref{t:MaiThm1}. The crucial step is Lemma \ref{l:MEigEst2} below,
which plays the same role as Lemma \ref{l:MEigEst} in the proof of Theorem \ref{t:MaiThm2}. With this lemma, we can prove Theorem \ref{t:MaiThm2} by the same argument as
showing Theorem \ref{t:MaiThm1}. So, in this section we only prove the crucial lemma but omit how to apply it to prove the theorem.
\begin{lem} \label{l:MEigEst2}
Assume that the conditions in Theorem \ref{t:MaiThm2} hold. For any $\gamma>0$ and $\ell \in (0,1/4)$, there exist some $\e_0>0$
depending on $\gamma$, $\ell$ and some $t_0 \in (0,1)$ depending on $\e_0$
such that $\lim_{\e_0 \rightarrow 0} t_0=0$ and that for all
$\e \in (0,\e_0)$ and $t \ge t_0$ we have
\Be \label{e:MEigEst2}
\PP\left(\lambda_{{\rm min}}(t) \le \e\right) \le C\re^{-c \left(\e^{\alpha \ell} |\log \e|^{\gamma}\right)^{-1}}.
\Ee
where $c$ only depends on $|A|$, $|B|$ and $C$ depends on $|A|$, $|B|$, $t$.
\end{lem}

To prove the above lemma, we need the following auxiliary lemma, which can be shown by an argument similar to proving Lemma \ref{l:UvUpB}.
\begin{lem} \label{l:UvUpB2}
Let $u, v \in \R^d$ both be nonzero vectors with some $p>0$ such that
\Be
\Ll v,u\Rr \ge p   \ \ ({\rm or} \ \Ll v,u\Rr \le -p).
\Ee
Then there exist some $\theta=\frac{1}{2|u||v|\|\nabla a\|_\infty}\re^{-\|\nabla a\|_\infty}$
$$\delta=\left(\theta p\right) \wedge 1$$ such that for all $t  \in (0,\delta)$ and $x \in \R^d$
\Be \label{e:UniUpB2}
\Ll K_t(x) v,u\Rr\ge p/2   \ \ ({\rm respectively} \ \Ll K_t(x) v,u\Rr\le-p/2).
\Ee
\end{lem}

\begin{proof} [Proof of Lemma \ref{l:MEigEst2}]
We first repeat exactly the steps 1 and 2 in the proof of Lemma \ref{l:MEigEst2}. To complete the proof, we only proceed to prove the
step 3.

Recall that the step 3 is to show that for any $\gamma>0$ and $\ell \in (0,1/4)$, there exist some $\e_0>0$
depending on $\gamma$, $\ell$ and some $t_0 \in (0,1)$ depending on $\e_0$
such that $\lim_{\e_0 \rightarrow 0} t_0=0$ and that for all
$\e \in (0,\e_0)$ and $t \ge t_0$ we have
\Be \label{e:Tar22}
\begin{split}
\PP\left(\Lambda(t,u,\e^{\ell})\le 2\e\right) \le \re^{-c|\log \e|^{-\gamma} \nu(\e^{\ell} \le |z| \le 1)}
\end{split}
\Ee
for all $u \in \mathbb{S}^{d-1}$, where $c>0$ only depends on $|A|$ and $|B|$.

By the uniform H\"{o}rmander condition, we have some $\kappa_0>0$ such that
\Be \label{e:HorTra}
\inf_{|u|=1} \sum_{i=1}^d \left(|\Ll \nabla a(x) B_i,u\Rr|^2+|\Ll B_i,u\Rr|^2\right) \ge 2d \kappa^2_0.
\Ee
Write $\theta=\frac{\re^{-\|\nabla a\|_\infty}}{2\|\nabla a\|_\infty |B|}$, we choose an $\e_0 \in (0,1/4)$ such that
\Be \label{e:E012}
|\log \e_0|^{-\gamma} \le \kappa_0,
\Ee
\Be \label{e:E022}
|\log \e_0|^{-2\gamma} h(\e^{\ell})>8 \e \ \ \ \ \ \forall \e \in (0,\e_0],
\Ee
\Be \label{e:E032}
\frac{8}{\kappa_0}|\log \e_0|^{-\gamma}<\left(\theta \kappa_0\right) \wedge 1.
\Ee
As $\e_0>0$ is sufficiently large, the above four conditions clearly hold.
Choosing
$$t_0=\max\{\theta |\log \e_0|^{-\gamma}, \frac{4}{\kappa_0}|\log \e_0|^{-\gamma}\},$$ we have
$t_0<1$ as $\e_0$ is sufficiently small.
We shall prove \eqref{e:Tar22} by considering the following two cases.

\emph{Case 1}: If there exists some $i_0 \in \{1,...,d\}$ such that
$$|\Ll B_{i_0}, u\Rr| \ge |\log \e_0|^{-\gamma},$$
choose
\Be \label{e:Delta2}
\delta=\theta |\log \e_0|^{-\gamma},
\Ee
it is clear that $\delta<t_0<1$ by the definition of $t_0$.
 By Lemma \ref{l:UvUpB2}, for all $x \in \R^d$ the following relation holds:
$$|\Ll K_s(x) B_{i_0}, u\Rr|\ge \frac 12 |\log \e_0|^{-\gamma} \ \ \ \  \forall s \in (0, \delta].$$
Write $N_{t,\e^{\ell}}=\int_0^t \int_{\e^{\ell} \le |z| \le 1} N(\dif z, \dif s)$, it follows from the above inequality and \eqref{e:SpeH}
that for all $\e \in (0,\e_0]$
\Be \label{e:UvUpB2}
\begin{split}
\int_0^\delta \int_{\e^{\ell} \le |z| \le 1} |\Ll K_s(x) B_{i_0}, u\Rr|^2 h(z) N(\dif z,\dif s)
& \ge \frac 14\int_0^\delta \int_{\e^{\ell} \le |z| \le 1} |\log \e_0|^{-2\gamma} h(z)  N(\dif z,\dif s)   \\
& \ge  \frac 14 |\log \e_0|^{-2\gamma} h(\e^{\ell}) N_{\delta, \e^{\ell}}.
\end{split}
\Ee
A straightforward computation gives
\Be \label{e:0Jum-1}
\PP(N_{\delta, \e^{\ell}}=0)=\re^{-\delta \nu(\e^{\ell} \le |z| \le 1)}.
\Ee
As $N_{\delta, \e^{\ell}} \ge 1$, \eqref{e:UvUpB2} and \eqref{e:E022} imply
\Bes
\int_0^\delta \int_{\e^{\ell} \le |z| \le 1} |\Ll K_s B_{i_0} v\Rr|^2 h(z) N(\dif z,\dif s) \ge
\frac 14 |\log \e_0|^{-2\gamma} h(\e^{\ell}) >2\e, \ \ \ \forall \e \in (0,\e_0].
\Ees
Hence,
\Be \label{e:1Jum-1}
\PP\left(\int_0^\delta \int_{\e^{\ell} \le |z| \le 1} |\Ll K_s B_{i_0}, u\Rr|^2 h(z) N(\dif z,\dif s) \le 2\e, N_{\delta, \e^{\ell}} \ge 1\right)=0.
\Ee
By \eqref{e:0Jum-1}, \eqref{e:1Jum-1} and the fact $t_0>\delta$ we have that for all $t \ge t_0$,
\Be \label{e:C1Inq}
\begin{split}
& \ \ \  \PP\left(\int_0^t \int_{\e^{\ell} \le |z| \le 1} |\Ll K_s B_{i_0},u\Rr|^2 h(z) N(\dif z,\dif s) \le
2\e\right)  \\
& \le \PP\left(\int_0^\delta \int_{\e^{\ell} \le |z| \le 1} |\Ll K_s B_{i_0},u\Rr|^2 h(z) N(\dif z,\dif s) \le
2\e\right)  \le \re^{-\delta \nu(\e^{\ell} \le |z| \le 1)}.
\end{split}
\Ee
By the definition of $\Lambda(t,u,\e^\ell)$ and $\theta |\log \e|^{-\gamma} \le \delta$,
the above inequality immediately implies the desired inequality \eqref{e:Tar22}.

\emph{Case 2}: If $|\Ll B_{i}, u\Rr| < |\log \e_0|^{-\gamma}$ for all $i \in \{1,...,d\}$,
by \eqref{e:HorTra} and \eqref{e:E012},
there exists some $i_1 \in \{1,...,d\}$ and some $\kappa_0>0$ so that
$$|\Ll \nabla a(x) B_{i_1},u\Rr| \ge \kappa_0 \ \ \ \forall x \in \R^d.$$
By Lemma \ref{l:UvUpB2}, as
$t \le \left(\theta \kappa_0\right) \wedge 1$, for all $\forall x \in \R^d$
the following relation holds:
\Bes
\Ll K_t\nabla a(x) B_{i_1},u\Rr \ge \kappa_0/2  \ \ {\rm or}\ \ \Ll K_t\nabla a(x) B_{i_1},u\Rr \le -\kappa_0/2.
\Ees
Therefore,
\Bes
\left|\int_0^t \Ll K_s\nabla a(x) B_{i_1},u\Rr \dif s\right| \ge \kappa_0 t/2 \ \ \ \ \forall t \le (\theta \kappa_0) \wedge 1.
\Ees
Choose
$$\tl \delta=\frac{8}{\kappa_0}|\log \e_0|^{-\gamma},$$
thanks to \eqref{e:E032}, the previous inequality, together with the easy relation
$$K_t(x) B_{i_1}=B_{i_1}-\int_0^t K_s(x) \nabla a(X_s) B_{i_1} \dif s, \ \ \forall x \in \R^d$$
implies
$$|\Ll K_t(x) B_{i_1}, u\Rr| \ge \kappa_0 t/2-|\log \e_0|^{-\gamma} \ge |\log \e_0|^{-\gamma} \ \  \ \ \forall t \in [\tl\delta/2, \tl \delta] \ \ \forall x\in \R^d.$$
By the same argument as in Case 1, we have
\Bes
\PP\left(\int_{\tl \delta/2}^{\tl \delta} \int_{\e^{\ell} \le |z| \le 1} |\Ll K_t B_{i_0},u\Rr|^2 h(z) N(\dif z,\dif s) \le 2\e\right) \le
\re^{-\frac{\tl \delta}2 \nu(\e^{\ell} \le |z| \le 1)} \ \ \ \forall \e \in (0,\e_0],
\Ees
By the same arguments as those below \eqref{e:C2Inq}, we get the desired inequality.
\end{proof}

\section{Appendix: The sketchy proof of \eqref{e:IBPGen}}
{\bf Step 1}: Define
$v^{\e}(z,t)=z+\e v(z,t)$, as $\e>0$ is sufficiently small $v^\e(.,t)$ as a function from $\R^d$ to $\R^d$ has a unique inverse. We denote this inverse by
$u^\e(z,t)$. Under our assumption it is easy to see that as $\e>0$ is sufficiently small,
\Be \label{e:UeZEst}
|u^\e(z,t)-z| \le C\e |z|^4 1_{|z| \le 2}(z).
\Ee
Further define
$$N^\e(\Gamma \times [0,t])=\int_0^t \int_{\R^d_0} 1_{\Gamma}(v^\e(z,s)) N(\dif z, \dif s),$$
it is easy to check that $N^\e$ has an intensity measure $\nu^\e$ satisfying
$$\nu^\e(\Gamma \times[0,t])=\int_0^t \int_{\R^d_0} 1_{\Gamma}(v^\e(z,s)) \nu(\dif z) \dif s.$$
As $\e>0$ is sufficiently small, the following Radon-Nikodym derivative always exists under our assumptions. A straightforward calculation gives
$$\frac{\dif \nu^\e}{\dif \nu \dif t}(z,t)=\frac{\rho(u^\e(z,t))}{\rho(z)}=:\varphi^\e(z,t),$$
where $\rho$ is the density function of $\nu$ defined in (H1).

{\bf Step 2}: Consider the following SDEs,
\Be \label{e:ZvEqn}
\dif Z^{\e}_t=(\varphi^\e(z,t)-1) \tl N(\dif z, \dif t), \ \ \ \ Z^\e_0=1,
\Ee
by It$\hat{o}$ formula we have
\Be \label{e:Zet}
Z^\e_t=\exp\{\int_0^t \log \varphi^\e(z,s) N(\dif z,\dif s)-\int_0^t  (\varphi^\e(z,s)-1) \nu(\dif z) \dif s\}.
\Ee
It is easy to check that $Z^\e_t$ is a martingale under our assumptions. Define a measure $\PP^{\e}$ which is determined by
\Be \label{e:RNDerP}
\frac{\dif \PP^\e}{\dif \PP}\bigg|_{\mcl F_t}=Z^\e_t, \ \ \ t>0,
\Ee
Thanks to \eqref{e:UeZEst}, we have
\Be \label{e:IntLim2}
\lim_{\e \rightarrow 0} \E |Z^\e_t|^2<\infty,
\Ee
\Be \label{e:IntLim3}
\lim_{\e \rightarrow 0} \E \left|\frac{(Z^\e_t)^{-1}-1}{\e}-\delta(V)\right|^2=0,
\Ee
where $\delta(V)$ is defined by \eqref{e:SkoInt}.
Define a process
$$L^\e_t=L_t+\int_0^t \int_{\R^d_0} \e v(z,s) N (\dif z,\dif s).$$
The crucial Girsanov type lemma holds:
\begin{lem} \label{l:Gir1}
The law of the process $(L^\e_t)_{t \ge 0}$ under $\PP^\e$ is the same as that of the process $(L_t)_{t \ge 0}$ under
$\PP$.
\end{lem}
\begin{proof}
By checking the characteristic functions of the arbitrary finite discretization of $L_t$ under $\PP$ and that of
$L^{\e}_t$ under $\PP^\e$ or by referring to \cite[Theorem 6.16]{BiGrJa86}.
\end{proof}

{\bf Step 3}:Consider the SDEs
\Bes
\begin{cases}
\dif X^\e_t=a(X^\e_t)\dif t+B \dif L^\e_t, \\
X^\e_0=x,
\end{cases}
\Ees
where $L^\e_t=L_t+\e \int_0^t \int_{\R^d_0} v(z,s) N(\dif z, \dif s)$. By Lemma \ref{l:Gir1},
for all $t>0$ the law of $X^\e_t$
under $\PP^\e$ and the law of $X_t$ under $\PP$ are the same.
Hence, for all $f \in C^1_b$,
\begin{align*}
\E\left[D_V f(X_t)\right]&=\lim_{\e \rightarrow 0} \E\left(\frac{f(X^\e_t)-f(X_t)}{\e}\right) \\
&=\lim_{\e \rightarrow 0} \frac1{\e} \left(\E f(X^\e_t)-\E^\e f(X^\e_t)\right) \\
&=\lim_{\e \rightarrow 0}  \E^\e \left[f(X^\e_t) \frac1{\e} \left(\frac{\dif \PP}{\dif \PP^\e}-1\right)\right] \\
&=\lim_{\e \rightarrow 0}  \E\left[f(X^\e_t) \frac{(Z^\e_t)^{-1}-1}{\e} Z^\e_t\right].
\end{align*}
where the first equality is thanks to $f \in C^1_b(\R^d)$ and the fact that $\frac{X^\e_t-X_t}{\e}$ is uniformly integrable.
By \eqref{e:IntLim2} and \eqref{e:IntLim3}, the above relation immediately gives the desired formula \eqref{e:IBPGen}. By a classical extension
procedure, we can show that the formula \eqref{e:IBPGen} also holds for $f \in C_b(\R^d)$.


\bibliographystyle{amsplain}

\end{document}